\documentclass[pdflatex,sn-mathphys]{JORSC-jnl}

\usepackage[T1]{fontenc}
\usepackage[utf8]{inputenc}
\usepackage{microtype}
\usepackage{nicefrac}
\usepackage{array}
\definecolor{orcidcolor}{HTML}{A6CE39}
\newcommand{\orcid}[1]{%
  \href{https://orcid.org/#1}{\textsuperscript{\textcolor{orcidcolor}{\textsf{\bfseries iD}}}}%
}

\makeatletter
\def\@acknotetext{}
\newcommand{\acknote}[1]{\gdef\@acknotetext{#1}}
\let\sn@orig@printabstract\printabstract
\def\printabstract{%
  \ifx\@acknotetext\@empty\else
    \par\addvspace{12pt}%
    \begingroup
      \leftskip=24pt\rightskip=24pt%
      \parfillskip=0pt plus 1fil\relax
      \small\noindent\@acknotetext\par
    \endgroup
    \addvspace{12pt}%
  \fi
  \sn@orig@printabstract%
}
\makeatother

\algrenewcommand\algorithmicrequire{\textbf{Input:}}
\algrenewcommand\algorithmicensure{\textbf{Output:}}

\theoremstyle{thmstyleone}
\newtheorem{theorem}{Theorem}[section]
\newtheorem{lemma}[theorem]{Lemma}

\newtheorem{prop}[theorem]{Proposition}

\theoremstyle{thmstylethree}
\newtheorem{definition}[theorem]{Definition}
\newtheorem{remark}[theorem]{Remark}

\newtheorem{example}[theorem]{Example}

\newcommand{\E}{\mathbb{E}}

\begin{document}

\title[Promoting Fair Online Resource Allocation with Indivisible Units]{Promoting Fair Online Resource Allocation with Indivisible Units}

\author[1]{\fnm{Igor} \sur{Averbakh}\orcid{0000-0001-6168-4424}}\email{igor.averbakh@rotman.utoronto.ca}
\author*[2]{\fnm{Hong-Yi} \sur{Jiang}\orcid{0000-0001-9840-078X}}\email{hongyi.jiang@cityu.edu.hk}
\author[3]{\fnm{Samitha} \sur{Samaranayake}\orcid{0000-0002-5459-3898}}\email{samitha@cornell.edu}
\author[4,5]{\fnm{A-Kang} \sur{Wang}\orcid{0000-0002-3325-8441}}\email{wangakang@sribd.cn}
\author[4]{\fnm{Jiang-Hua} \sur{Wu}\orcid{0000-0002-1394-6618}}\email{wujianghua@sribd.cn}

\affil[1]{\orgdiv{Department of Management, University of Toronto Scarborough, and Rotman School of Management}, \orgname{University of Toronto}, \orgaddress{\city{Toronto}, \country{Canada}}}
\affil[2]{\orgdiv{Department of Systems Engineering}, \orgname{City University of Hong Kong}, \orgaddress{\city{Hong Kong}, \country{China}}}
\affil[3]{\orgdiv{School of Civil and Environmental Engineering}, \orgname{Cornell University}, \orgaddress{\city{Ithaca}, \state{NY}, \country{USA}}}
\affil[4]{\orgname{Shenzhen Research Institute of Big Data}, \orgaddress{\city{Shenzhen}, \country{China}}}
\affil[5]{\orgdiv{School of Data Science}, \orgname{The Chinese University of Hong Kong, Shenzhen}, \orgaddress{\city{Shenzhen}, \country{China}}}

\abstract{Allocating scarce, indivisible resources to diverse groups under uncertainty is a central challenge in operations research, where efficiency-focused methods often underserve marginalized populations. We study the Fair Online Resource Allocation with Indivisible Units (FORA-IU) problem, in which an unpredictable sequence of demands must be served from a strictly fixed inventory, and ask what fairness guarantees are achievable under different distributional and structural assumptions. 
We adopt a fairness criterion based on the expected filling ratio ($FE\text{-}FR\text{-}\beta$), which balances each group's expected allocation against its expected demand and priority weight. We design online policies that calibrate acceptance probabilities to the remaining budget, analyze both arbitrary time-varying and stationary arrivals, introduce the Random Cyclic Blocks (RCB) algorithm tailored to the stationary case, and study the effect of restricting policies to all-or-nothing allocations. 
For arbitrary time-varying arrivals, our policy achieves the optimal universal fairness guarantee of $\frac{1}{1+R_\beta}$, where $R_\beta$ denotes the priority-weighted system load. For time-invariant arrivals, RCB achieves the exact finite-horizon guarantee $[1-(1-R_\beta/T)^T]/R_\beta$, which is at least $(1-e^{-R_\beta})/R_\beta$ and is also tight. We further show that all-or-nothing allocation policies cannot match these guarantees. 
These findings demonstrate that distributional stationarity strictly improves the fairness frontier, and that partial fulfillment is a necessary condition for attaining optimal fairness in online indivisible resource allocation.}



\acknote{Igor Averbakh was supported by the Natural Sciences and Engineering Research Council of Canada Discovery Grant (No. RGPIN-2018-05066). \\Hong-Yi Jiang was supported by the National Natural Science Foundation of China (No. 72501239). \\A-Kang Wang was supported by the National Natural Science Foundation of China (No. 12301416).\\The authors thank Prof. Pan Xu from the New Jersey Institute of Technology for helpful discussions at an early stage of this work.}

\maketitle

\section{Introduction}

Allocating scarce resources under uncertainty is a central problem in operations research. Over the past decade, online resource allocation problems have attracted significant interest. In these settings, requests arrive sequentially with stochastic demands, and the allocator must make immediate decisions without knowledge of future arrivals. Such problems arise in diverse domains, from cloud computing \cite{du2022adversarial} to humanitarian logistics \cite{manshadi2023fair}. Traditional approaches prioritize efficiency, aiming to maximize total utility over the planning horizon. However, growing evidence shows that efficiency-focused policies, when deployed without explicit fairness constraints, tend to amplify existing societal inequities \cite{obermeyer2019dissecting}.

The mechanism behind this inequity is intuitive. Algorithms that optimize for throughput naturally favor requests that are easiest to serve. In practice, this means prioritizing groups with better access to infrastructure, information, or flexibility. Marginalized populations often lack the ability to submit requests quickly or at favorable times. Consequently, they are underserved \cite{samorani2022overbooked}. The resulting disparities are not merely theoretical. They produce measurable and severe outcomes. Studies of post-disaster relief show that resource distribution often widens pre-existing wealth gaps between demographic groups \cite{howell2019damages}. Similarly, during public health emergencies, vulnerable populations frequently receive disproportionately small shares of critical medical supplies despite facing higher mortality risks \cite{ne2022identifying}. These findings underscore a critical point. Different groups have fundamentally different levels of urgency. Achieving equitable outcomes requires allocation mechanisms that actively enforce proportional service guarantees.

Designing such mechanisms, however, faces a fundamental technical barrier. Many resources of practical interest cannot be subdivided arbitrarily. Ventilators, food boxes, and housing units must be allocated as whole items, unlike bandwidth or monetary transfers. This \emph{discreteness} creates an inherent tension with fairness objectives. When two groups compete for a single indivisible unit, the allocator must grant the entire unit to one group while denying the other completely. No intermediate compromise is possible within a single allocation instance.

Another key observation is that many real-world allocation scenarios are not isolated events but repetitive processes. A food bank distributes resources daily. A disaster relief agency operates across multiple distribution cycles. In such settings, we can leverage the Law of Large Numbers to shift our fairness objective from individual realizations to long-run expectations. By designing policies that achieve fairness in expectation, we ensure that realized allocations converge to their equitable shares as the number of allocation episodes grows. This approach bridges the gap between the rigid physical constraints of indivisibility and the systemic equity required to serve diverse populations fairly over time.

Motivated by these considerations, we study the \emph{Fair Online Resource Allocation with Indivisible Units} (FORA-IU) problem. Our framework assigns priority weights to each group, encoding heterogeneous urgency levels. Higher-risk populations receive service guarantees proportional to their weighted importance. We also depart from standard all-or-nothing constraints, where requests are rejected entirely if they cannot be fully satisfied. Instead, our approach permits \textit{partial allocation}. By allowing the fulfillment of a subset of each request, we maximize the achievable fairness guarantee. This prevents high-priority groups from being penalized simply because their exact requests do not perfectly match the available inventory.

\subsection{Problem Formulation}

We now formalize the FORA-IU problem. Consider a resource allocator managing a budget of $K$ indivisible units (e.g., complete medical kits or pallets of water) over a finite horizon of $T$ time slots (e.g., days in a disaster response period). The population served by the allocator is partitioned into $N$ disjoint groups, which could represent separate relief zones or different levels of need, such as the elderly versus the general public.

To capture the heterogeneous urgency discussed above, we assign a priority weight $\beta_i \in (0, 1]$ to each group $i \in [N]$. Without loss of generality, we normalize these weights so that $\max_{i \in [N]} \beta_i = 1$.

At each time slot $t \in [T]$, the allocator receives and handles at most one resource request. The probability that a request arrives from group $i$ during slot $t$ with a demand of $j \in [K]$ units is denoted by $p_{itj} \geq 0$. Arrivals are independent across time slots, and the total arrival probability in any slot satisfies $\sum_{i=1}^N\sum_{j=1}^K p_{itj} \leq 1$. We let $S_{itj}$ denote the indicator random variable equal to $1$ if a request from group $i$ with demand $j$ arrives at time $t$, and $0$ otherwise.

Upon the arrival of a request, the allocator must immediately decide the quantity $C_{it}$ of resources to allocate. Three constraints govern this decision, reflecting the physical properties of the resources and the online nature of the problem.
\begin{enumerate}
    \item $C_{it}$ must be determined without knowledge of future arrivals.
    \item $C_{it}$ must be a nonnegative integer. Crucially, we permit \textit{partial allocation}, meaning the allocator may fulfill any subset of the demand, so that $0 \leq C_{it} \leq \sum_{j=1}^K S_{itj} \cdot j$.
    \item The total allocation $\sum_{t=1}^T \sum_{i=1}^N C_{it}$ cannot exceed the budget $K$.
\end{enumerate}

We measure the fairness of an allocation policy using the \textbf{Fair Expected Filling Ratio with priority coefficients} (FE-FR-$\beta$). A (randomized) policy achieves an FE-FR-$\beta$ of at least $\alpha \in [0,1]$ if, for every group $i \in [N]$,
\begin{align}\label{eq:UE-FR-beta}
\E\left[\sum_{t=1}^T C_{it}\right] \geq \alpha \cdot \beta_i \cdot \E\left[\sum_{t=1}^T \sum_{j=1}^K S_{itj} \cdot j\right].
\end{align}
This metric guarantees that a group with priority $\beta_i$ receives, in expectation, at least a fraction $\beta_i \alpha$ of its expected demand.

To quantify the scarcity of resources relative to demand, we define the \emph{priority-weighted load parameter} $R_\beta$ as the ratio of total expected weighted demand to the available capacity $K$,
\begin{align}\label{eq:R-beta}
    R_\beta := \frac{1}{K} \sum_{t=1}^T \sum_{i=1}^N \sum_{j=1}^K \beta_i \cdot p_{itj} \cdot j.
\end{align}
This parameter measures the \textbf{effective system load}, capturing the aggregate pressure on the budget after discounting requests from lower-priority groups.

Our objective is to design an online policy that \emph{maximizes $\alpha$} as a function of the instance parameter $R_\beta$.

\subsection{Justification for the Expected Fairness Metric}

A natural alternative is to use the ex-ante minimum fill-rate notion of \citet{manshadi2023fair}, specialized to our group-level setting. Since group \(i\) may generate multiple requests over the horizon, define its realized fill rate by
\[
F_i :=
\begin{cases}
1, & D_i=0,\\[2mm]
\dfrac{A_i}{D_i}, & D_i>0,
\end{cases}
\qquad
\text{where }
A_i := \sum_{t=1}^T C_{it},
\quad
D_i := \sum_{t=1}^T\sum_{j=1}^K S_{itj}\cdot j.
\]
The corresponding ex-ante metric is
\(
\min_{i\in[N]} \E[F_i].
\)

We argue that this metric is not suitable in our setting. The difficulty is that, by convention, every realization with \(D_i=0\) contributes a value of \(1\). Hence a group that is never served when it actually has positive demand may still appear almost perfectly treated if its total demand is zero with high probability.

This pathology already appears in a simple unweighted instance with \(\beta_i=1\) for all \(i\). Let \(N=2\), and suppose
\(
D_1=1 
\) almost surely,
while
\[
D_2=
\begin{cases}
K-1, & \text{with probability } \frac{1}{K-1},\\
0, & \text{with probability } 1-\frac{1}{K-1}.
\end{cases}
\]
Then
\[
\E[D_1]=1,
\qquad
\E[D_2]=\frac{1}{K-1}\cdot (K-1)=1,
\]
so the two groups have the same expected demand. Now consider the feasible policy \(\pi\) that allocates one unit to group \(1\) and zero units to group \(2\). Under the ex-ante fill-rate metric from \cite{manshadi2023fair},
\[
\E[F_1]_{\pi}=1,
\qquad
\E[F_2]_{\pi}
=
\left(1-\frac{1}{K-1}\right)\cdot 1
+
\frac{1}{K-1}\cdot 0
=
1-\frac{1}{K-1}.
\]
Hence
\[
\min_{i\in[2]} \E[F_i]_{\pi}
=
1-\frac{1}{K-1}
\to 1
\qquad \text{as } K\to\infty.
\]
Thus this metric certifies the policy as nearly perfectly fair, even though group \(2\) is never served.

By contrast, under Eq.~\eqref{eq:UE-FR-beta},
\[
\frac{\E[A_1]}{\E[D_1]}=1,
\qquad
\frac{\E[A_2]}{\E[D_2]}=0.
\]
So the expectation-based metric correctly identifies that one group receives none of its expected demand despite having the same expected demand as the other group. For this reason, Eq.~\eqref{eq:UE-FR-beta} is the more appropriate fairness notion in our stochastic online setting.

One could also ask for the stronger ex-post requirement
\[
A_i(\omega)\ge \alpha\,\beta_i\,D_i(\omega)
\qquad \text{for every realization } \omega.
\]
This is even less suitable. For example, take \(K=1\), \(T=2\), \(\beta_1=\beta_2=1\), and suppose group \(1\) requests one unit at \(t=1\) while group \(2\) requests one unit at \(t=2\). Then \(D_1=D_2=1\), but the total budget is only \(1\). Hence, for any \(\alpha>0\), no policy can satisfy the ex-post requirement for both groups simultaneously.

The expectation-based metric overcomes this limitation by exploiting the repetitive structure inherent in FORA-IU applications. Consider a food bank that distributes a fixed budget of $K$ food boxes daily. Here, the finite horizon $T$ corresponds to the operating time slots of a single day, and the food bank operates over a sequence of $M$ independent days. On any given day $m$, the realized allocation $A_i^{(m)}$ might be zero for a high-priority group due to the indivisibility constraints described above. However, the relevant fairness criterion concerns cumulative service over the long run.

Let $(A_i^{(m)}, D_i^{(m)})$ denote the allocation and demand totals on day $m$. By the Strong Law of Large Numbers, the empirical averages converge almost surely to their expectations. Provided $\E[D_i] > 0$, we have
\[
\frac{\sum_{m=1}^M A_i^{(m)}}{\beta_i\sum_{m=1}^M D_i^{(m)}}
=
\frac{\frac{1}{M}\sum_{m=1}^M A_i^{(m)}}{\beta_i\cdot \frac{1}{M}\sum_{m=1}^M D_i^{(m)}}
\xrightarrow{\mathrm{a.s.}}
\frac{\E[A_i]}{\beta_i\E[D_i]}
\;\ge\;\alpha.
\]
Thus, the FE-FR-$\beta$ constraint delivers an \emph{asymptotic ex-post guarantee} on the realized long-run filling ratio without requiring an arbitrary and problematic slack parameter $\Delta$.

This long-run perspective is consistent with recent literature on fairness in online resource allocation. \citet{ma2024promoting} adopt the same philosophy in their long-run fairness notion FAIR-L for online bipartite matching, defined as $\min_{j \in J} \E[X_j]/\E[A_j]$. They treat each horizon as a single day and audit the algorithm across many independent days, so that the empirical fairness converges almost surely to this ratio of expectations. They likewise observe that any deterministic policy collapses to zero short-run (ex-post) fairness at peak load, mirroring the indivisibility dilemma we describe above. The two metrics nevertheless differ in scope. FAIR-L governs single-unit matching with all online types treated symmetrically, whereas our FE-FR-$\beta$ handles multi-unit indivisible allocation with partial fulfillment and embeds heterogeneous priority weights $\beta_i$ directly into the fairness target. This yields a load-parameterized optimal guarantee that adapts to the effective system load $R_\beta$ without invoking arbitrary slack parameters that become vacuous under small demands.

\section{Contributions}

We study FORA-IU, an online demand-rationing problem with a hard budget of indivisible units, stochastic multi-unit requests, partial fulfillment, and heterogeneous priority weights. Our main contribution is a tight characterization of the best achievable FE-FR-$\beta$ guarantee as a function of the priority-weighted load $R_\beta$. Specifically, our contributions are threefold.

\begin{enumerate}
    \item For arbitrary time-varying arrival distributions $(p_{itj})$, we design an online policy (Algorithm~\ref{alg:weighted-fora}) that calibrates randomized acceptance through the state-dependent quantity
    \(
    \gamma_{tj}=\mathbb{E}[\min(B_t/j,1)],
    \)
    where $B_t$ is the residual budget under the policy itself and $\gamma_{tj}$ is computed by a forward dynamic program. To analyze this policy, we introduce a strengthened per-arrival criterion, the \emph{Refined Fair Expected Filling Ratio} (Definition~\ref{def:RFE-FR}), and show that the policy attains the optimal FE-FR-$\beta$ guarantee of $\frac{1}{1+R_\beta}$ for this class (Theorem~\ref{thm:weighted-fora}). A matching lower bound (Theorem~\ref{thm:optimal-rFE-FR-beta-active}) shows that no online policy can do better on this class of instances.

\item For the time-invariant case, we propose the \emph{Random Cyclic Blocks} (RCB) algorithm (Algorithm~\ref{alg:weighted-rcb}), which places the $K$ indivisible units on a cycle and serves each request by claiming a uniformly rotated contiguous block. A rotational-symmetry argument reduces the analysis to a single unit and yields the load-only FE-FR-$\beta$ guarantee
\(
\frac{1-e^{-R_\beta}}{R_\beta},
\)
which strictly improves on the optimal universal guarantee $\frac{1}{1+R_\beta}$. A sharper exact finite-horizon characterization and a matching upper bound are established later in the section, showing that this improvement is precisely due to time-invariant arrivals.

    \item We prove a structural separation between allocation protocols. Through explicit hard instances, Theorem~\ref{thm:limit-all-or-nothing} shows that policies restricted to all-or-nothing allocations cannot match either of our optimal guarantees, in both the general and the time-invariant settings. Partial fulfillment is therefore essential rather than a modeling convenience.
\end{enumerate}

\section{Related Work}\label{sec:related}

Resource allocation has a long history in offline and static settings. Classical models allocate a fixed
amount of resources across competing activities, typically with continuous or integer decision variables,
and have been surveyed from algorithmic and nonlinear-optimization perspectives
\citep{patriksson2008survey}. This literature also includes fair and decentralized variants, such as
multiagent resource allocation \citep{chevaleyre2006issues} and game-theoretic allocation in urban
systems \citep{krylatov2024competitive}. These works provide broad background on resource allocation,
but they do not address the online arrivals, stochastic demand, and indivisible supply that are central to
FORA-IU.

Sequential allocation settings, in which requests or items arrive over time and decisions are irrevocable,
are therefore more directly relevant to our model. Within this body of work, the literature closest to
FORA-IU lies at the intersection of group-based fairness in online allocation, sequential rationing under
uncertain demand, dynamic fair division, and a smaller set of models especially close to our fairness
objective.

Recent work on online matching and allocation has emphasized fairness across protected or otherwise salient groups, often through class-level constraints or group-specific performance guarantees~\citep{ijcai2021p53,hosseini2024class,ma2024promoting,zargari2025online}. Parallel work on societal platforms such as ridesharing and public transportation designs policies that satisfy explicit equity targets while preserving system objectives such as efficiency or profit~\citep{nanda2020balancing,pramanik2023equity,xu2024exploring}. This literature shares with our setting the view that fairness must be enforced dynamically during the allocation process rather than evaluated only ex post.

A complementary stream studies sequential rationing and stochastic service systems, where a limited budget must be distributed over a horizon to meet uncertain demand. Such models arise naturally in nonprofit and humanitarian operations, where fairness is often expressed through service levels or fill rates under horizon-wide feasibility constraints~\citep{lien2014sequential,fadaki2025sequential,wang2019measuring}. \citet{mertzanidis2024automating} develop an automated food-rescue system that matches arriving donation truckloads to food banks; in their model each truckload is an indivisible item that must be assigned upon arrival, whereas FORA-IU allocates from a shared homogeneous budget to arriving requests. In the algorithmic literature, fair dynamic rationing examines how a fixed inventory should be used over time under stochastic demand while maintaining equity-style guarantees, often by comparing realized service against benchmark allocations~\citep{balseiro2021regularized,manshadi2023fair}. Related work formulates sequential fair allocation as a Markov decision process and optimizes fairness-aware objectives under uncertainty~\citep{fianu2018markov,hassanzadeh2023sequential}. Other variants incorporate monetary transfers~\citep{fallah2024fair}, stochastic utilities~\citep{sinclair2022sequential}, or operational features such as perishing resources and replenishments~\citep{banerjee2023online,onyeze2025replenishments}. Our model instead focuses on integer multi-unit demands drawn from a fixed budget of indivisible units, allows partial fulfillment, and incorporates heterogeneous group priorities.

Our work is also related to online and dynamic fair division, where fairness is usually defined through proportionality- or envy-based notions as agents or items arrive over time. Representative models include online cake cutting for divisible resources~\citep{walsh2011online}, food-bank motivated online allocation~\citep{aleksandrov2015online}, and dynamic fair division with evolving populations or limited reassignment~\citep{kash2014no,he2019achieving,zeng2020fairness}. A large recent literature studies settings in which heterogeneous items arrive sequentially or agents arrive sequentially to claim items from an upfront pool~\citep{benade2024fair,halpern2025online,kulkarni2025online}. These models typically evaluate fairness through envy-based or share-based guarantees such as envy-freeness up to one good or maximin share, comparing bundles of distinct items across uniformly treated agents. By contrast, FORA-IU concerns stochastic requests from fixed groups with possibly different priority levels, each requesting multiple units from a shared homogeneous budget, and fairness is measured by priority-weighted fill rates rather than inter-agent envy over heterogeneous bundles. Several nearby models also differ from ours in other essential ways. \citet{cookson2025temporal} assume full upfront knowledge of future item arrivals, \citet{golz2025apportionment} study one-shot randomized seat apportionment rather than dynamic rationing, and \citet{elkind2024fair} impose budget constraints on agents rather than on the allocator.

These differences make FORA-IU closer to a demand-rationing model than to standard online fair division. Its distinctive features are that requests are stochastic, supply is a hard budget of indivisible units, requests may ask for multiple units, partial fulfillment is allowed, and fairness is defined at the group level with heterogeneous priority weights~$\beta_i$. To the best of our knowledge, this combination has not been studied previously. The closest technical comparisons therefore come from simplified settings that capture only part of the model.

A natural simplification is the unit-demand case in which every request has size $j=1$. When $\beta_i=1$ for all groups, FE-FR reduces to a group-wise acceptance-rate objective. \citet{srinivasan2022generalized} analyze budgeted online fairness within the generalized magician problem framework under unknown arrival distributions, where the arrival law is never revealed to the algorithm. For unknown i.i.d.\ arrivals, they show that a simple greedy algorithm that accepts any item whenever budget remains is optimal; for the magician special case, in which all items are Bernoulli, they derive an exact closed-form expression for the optimal fairness in terms of the budget $B$. For unknown adversarial distributions, where arrival probabilities may vary arbitrarily over time, they give a near-optimal non-adaptive guarantee that vanishes as $B$ grows. Their results therefore apply only to a special case of FORA-IU and do not address multi-unit demands, partial fulfillment, or heterogeneous priority weights. \citet{ma2024promoting} study online bipartite matching under known stationary Poisson arrivals, where the objective is to maximize the minimum service rate across online types. Each arriving agent has unit demand and offline agents have finite serving capacities. They show that first-come-first-served is optimal when there is a single offline agent, and for the general case they propose a sampling-based algorithm whose fairness guarantee is at least $1-1/e$ and approaches one in several practically relevant regimes. Even in special cases, however, the bipartite matching structure differs fundamentally from FORA-IU's shared-budget model with variable multi-unit requests.

Another prior work that is close to our fairness notion is \citet{SankararamanSrinivasanXu2024}, which studies external equity promotion with random nonnegative demands and fairness measured through filling ratios $X_i/D_i$, motivated by applications including refugee resettlement, public housing, emergency aid, food donation, medical supply allocation, ride-hailing, and vaccine distribution. Under independent demands and ex-ante fairness, they obtain LP-based policies achieving a competitive ratio of $1/2$ relative to a benchmark LP, together with additional unconditional results, including an optimal policy when each demand is either $0$ or at least $1$, and a $(1-\varepsilon)$-competitive policy under ex-post fairness when every demand is bounded by $K\varepsilon^2/\ln(1/\varepsilon)$ for an absolute constant $K$. The conceptual link to our work is clear, since both models use filling ratios as the fairness criterion, but the technical settings differ substantially. Their supply is infinitely divisible and normalized to one rather than a hard budget of indivisible units, they do not consider group priority weights $\beta_i$, and their guarantees are competitive ratios against a clairvoyant offline benchmark rather than absolute load-dependent fairness bounds. The two sets of results are therefore complementary rather than directly comparable.

Our algorithms are also related to existing randomized online policies at the methodological level. Algorithm~\ref{alg:weighted-fora} and \citet{MaXuXu2023} both pursue fairness maximization through randomized online policies, but the underlying mechanisms are quite different. \citet{MaXuXu2023} use an LP-based sampling framework with boosting and, in their strongest group-fairness algorithm, Monte Carlo simulation-based attenuation. Our time-varying policy instead applies a priority-based Bernoulli screening step followed by the direct threshold
\(
\frac{1}{(1+R_\beta)\gamma'_{tj}},
\)
where $\gamma'_{tj}=\mathbb{E}[\min(B'_t/j,1)]$ is precomputed exactly by dynamic programming on the virtual instance $\mathcal{I}'$. Algorithm~\ref{alg:rcb} is also related in spirit to the online correlated selection framework of \citet{FahrbachHuangTaoZadimoghaddam2022}, since both replace greedy allocation with structured randomization over a symmetric domain. The resemblance is only at a high level, however. OCS creates temporal negative correlation across rounds involving the same offline vertex, boosting the probability that a vertex appearing in $k$ pairs is matched at least once from $1-2^{-k}$ to $1-2^{-k}(1-\gamma)^{k-1}$ and breaking the long-standing $1/2$ competitive-ratio barrier for edge-weighted online bipartite matching under free disposal. RCB instead uses spatial symmetry on a cycle: a uniform random rotation makes every unit equally exposed to future requests, and first-come-first-served consumption of contiguous blocks yields group-level filling-ratio guarantees.

Taken together, the existing literature captures important parts of the setting but not the full model studied here. To the best of our knowledge, no prior work jointly handles fixed groups with heterogeneous priority weights $\beta_i$, stochastic multi-unit requests, indivisible supply, and partial fulfillment over a finite horizon. FORA-IU fills this gap and provides tight load-dependent fairness guarantees of $\frac{1}{1+R_\beta}$ for time-varying arrivals and a strictly stronger load-only guarantee of $\frac{1-e^{-R_\beta}}{R_\beta}$ for time-invariant arrivals, with a sharper exact finite-horizon characterization given by $\frac{1-(1-R_\beta/T)^T}{R_\beta}$.

\section{Optimal Online Policy for Time-Varying Arrivals}\label{sec:general} 

We propose an optimal online policy for the FORA-IU problem where arrival distributions $(p_{itj})$ vary arbitrarily over time slots. This policy achieves an FE-FR-$\beta$ of $\frac{1}{1+R_\beta}$, which we prove to be optimal.

The section is organized as follows. We begin with the special case of uniform priorities ($\beta_i = 1$ for all groups). We then extend the policy to handle heterogeneous priority weights through a natural reduction that filters requests based on their priority levels. Finally, we establish a matching lower bound, confirming that our algorithm is optimal.

\subsection{The Homogeneous-Priority Case}

We first consider the special case where all groups share the same priority, i.e., $\beta_i = 1$ for all $i \in [N]$. In this setting, the priority-weighted load reduces to $R := \frac{1}{K} \sum_{t,i,j} p_{itj} \cdot j$, and the fairness metric FE-FR-$\beta$ is denoted as FE-FR for brevity. Our goal is to construct an online policy that attains an FE-FR of $\frac{1}{1+R}$ for all problem instances with same priorities.

To construct this policy, we target a strictly stronger, granular fairness condition. Specifically, we ensure that fairness is maintained locally at every arrival (in expectation). We term this stronger condition the \textbf{Refined Fair Expected Filling Ratio} (RFE-FR).

\begin{definition}\label{def:RFE-FR}
Given a FORA-IU instance, a (randomized) online policy is said to achieve an RFE-FR of at least $\alpha \in [0, 1]$ if, for every $i \in [N]$, $t \in [T]$, and $j \in [K]$ with $p_{itj}>0$,
\(
\mathbb{E}\left[C_{it} \mid S_{itj} = 1\right] \geq \alpha \cdot j,
\)
where the expectation is taken over the randomness of both the arrival process and the policy, conditioned on the specific arrival event $S_{itj}=1$.
\end{definition}

While FE-FR evaluates the expected allocation aggregated over the entire horizon, RFE-FR requires that the expected allocation be proportional to the demand \emph{conditional on each arrival event with positive probability}. Because the total expected allocation is the sum of these conditional expectations weighted by the corresponding arrival probabilities, the transition from local to global fairness is immediate.

\begin{lemma}\label{lem:sufficiency}
If an online policy attains an RFE-FR of $f(R)$ for a FORA-IU instance, it also achieves an FE-FR of at least $f(R)$ for that instance.
\end{lemma}

The proof follows directly from the law of total expectation after summing over arrival events with positive probability, and is omitted.

Guided by this sufficiency result, we propose Algorithm~\ref{alg:non-iid}. The algorithm dynamically adjusts acceptance probabilities based on the distribution of the remaining budget, explicitly targeting the RFE-FR condition. First, we give an informal sketch of the policy and of the computation of policy-defining values. Once a request arrives in time slot $t$ from group $i$ with demand $j$, with probability $1 - \frac{1}{(1+R)\gamma_{tj}}$ the request is denied, and with probability $\frac{1}{(1+R)\gamma_{tj}}$ it is fulfilled as much as possible given the current remaining budget $B_t$ (that is, in the amount $\min (B_t, j)$). Here,  $\gamma_{tj}$ are values that are pre-computed in advance (before observing any realized demand) for all $t\in [T]$ and $j\in [K]$ and are not adjusted later based on the realized demand. This pre-computation is done by the dynamic programming Algorithm~\ref{alg:dynamic-beta} (see Appendix~\ref{sec:beta-calc}), starting with the boundary condition $B_1=K$ (hence $\gamma_{1k}=1$ for all $k\in[K]$) and expanding it consecutively for $t=2,3,\ldots,T$ to ensure that $\gamma_{tk}=\mathbb{E}[\min(B_t/k,1)]$, where $B_t$ denotes the remaining budget at the beginning of slot $t$ and the expectation is over the randomness of both the arrival process and the randomized fulfillment policy described above. This dynamic program also pre-computes the state probabilities $P(B_t=b)$ for $b\in\{0,\ldots,K\}$. 

\begin{algorithm}[h]
\caption{FORA-IU Algorithm for Unit Priorities}\label{alg:non-iid}
\begin{algorithmic}[1]
\Require Parameters $K$, $N$, $T$, and arrival probabilities $p_{itj}$. 
\Ensure Resource allocation decision $C_{it}$ for each time slot.
\State Pre-calculate $\gamma_{tj} := \mathbb{E}[\min(B_t / j, 1)]$ for each $t \in [T]$ and $j \in [K]$ using Algorithm~\ref{alg:dynamic-beta} (see Appendix~\ref{sec:beta-calc}).
\State Initialize $B_1 := K$ and $C_{it} := 0$ for all $i, t$.
\For{$t = 1$ \textbf{to} $T$}
    \If{a request arrives from group $i$ with demand $j$ ($S_{itj} = 1$)}
        \State Set the allocation $C_{it}$ according to the following randomized rule:
        \begin{align}\label{eq:non-iid-allocate-prob}
        C_{it} =
        \begin{cases}
            \min(B_t, j) & \text{with probability } \frac{1}{(1+R)\gamma_{tj}}, \\[0.5ex]
            0 & \text{with probability } 1 - \frac{1}{(1+R)\gamma_{tj}}.
        \end{cases}
        \end{align}
        \State (Note: Theorem~\ref{thm:non-iid} proves that $\gamma_{tj} \geq \frac{1}{1+R}$, ensuring this probability is valid.)
    \EndIf
    \State Set $B_{t+1} := B_t - \sum_{i=1}^N C_{it}$.
\EndFor
\end{algorithmic}
\end{algorithm}

\begin{theorem}\label{thm:non-iid}
Algorithm~\ref{alg:non-iid} is well-defined (i.e., $\gamma_{tj} \geq \frac{1}{1+R}$ for all $t, j$) and attains an RFE-FR of $\frac{1}{1+R}$. Consequently, by Lemma~\ref{lem:sufficiency}, it achieves an FE-FR of $\frac{1}{1+R}$.
\end{theorem}

\begin{proof}
We proceed in two steps. First, we prove that the algorithm is well-defined by showing that the probability term $\frac{1}{(1+R)\gamma_{tj}} \in [0,1]$. This is equivalent to showing $\gamma_{tj} \geq \frac{1}{1+R}$ for all $t, j$. Second, we verify that this condition implies the desired RFE-FR guarantee.

\medskip
\noindent
We use induction on $t$ to establish $\gamma_{tj} \geq \frac{1}{1+R}$.

\begin{itemize}
    \item \textbf{Base Case ($t=1$):} At the start, the budget is full, so $B_1 = K$ with probability 1. Thus, $\gamma_{1j} = \mathbb{E}[\min(K/j, 1)] = 1$ (since $j \le K$). Since $R \ge 0$, it holds that $1 \ge \frac{1}{1+R}$.

    \item \textbf{Inductive Step:} Assume that  for every $\tau < t$ and every $j \in [K]$, we have $\gamma_{\tau j} \geq \frac{1}{1+R}$, so the probability term $\frac{1}{(1+R)\gamma_{\tau j}}$ is valid. Then for every prior triple $(i,\tau,j)$ with $p_{i\tau j}>0$, conditioned on the event $S_{i\tau j}=1$, the algorithm allocates $\min(B_\tau,j)$ with probability $\frac{1}{(1+R)\gamma_{\tau j}}$ and $0$ otherwise. Since $B_\tau$ depends only on arrivals and allocations in slots $1,\ldots,\tau-1$, and arrivals are independent across time slots, $B_\tau$ is independent of the arrival event $S_{i\tau j}$ at slot $\tau$. Therefore,
\begin{align*}
        \mathbb{E}[C_{i\tau} \mid S_{i\tau j}=1]
    &= \frac{1}{(1+R)\gamma_{\tau j}} \cdot \mathbb{E}[\min(B_\tau,j)\mid S_{i\tau j}=1]\\
    &= \frac{1}{(1+R)\gamma_{\tau j}} \cdot \mathbb{E}[\min(B_\tau,j)]
    = \frac{j}{1+R},
\end{align*}
where the last equality uses the identity $\mathbb{E}[\min(B_\tau,j)] = j \cdot \gamma_{\tau j}$.
\end{itemize}

Now consider time slot $t$. Recall that $\gamma_{tj} := \mathbb{E}[\min(B_t/j, 1)]$. We observe two properties:
\begin{enumerate}
    \item Since $j \leq K$, we have $B_t/j \geq B_t/K$. Consequently, $\min(B_t/j, 1) \geq \min(B_t/K, 1)$.
    \item Since the remaining budget $B_t$ never exceeds the initial capacity $K$, $\min(B_t/K, 1) = B_t/K$.
\end{enumerate}
Combining these, we get $\gamma_{tj} \geq \mathbb{E}[B_t]/K$.

Next, we expand $\mathbb{E}[B_t]$. The expected remaining budget at time $t$ is the initial budget $K$ minus the expected consumption up to time $t-1$:
\begin{align*}
    \mathbb{E}[B_t] &= K - \sum_{\tau=1}^{t-1} \sum_{i=1}^N \mathbb{E}[C_{i\tau}] 
    = K - \sum_{\tau=1}^{t-1} \sum_{i=1}^N \sum_{\substack{k \in [K]\\ p_{i\tau k}>0}} P(S_{i\tau k}=1)\cdot \mathbb{E}[C_{i\tau}\mid S_{i\tau k}=1] \\
    &= K - \sum_{\tau=1}^{t-1} \sum_{i=1}^N \sum_{\substack{k \in [K]\\ p_{i\tau k}>0}} p_{i\tau k}\cdot \frac{k}{1+R}
    \quad \text{(by the induction hypothesis)}\\
    &\geq K - \sum_{\tau=1}^{T} \sum_{i=1}^N \sum_{k=1}^K p_{i\tau k} \cdot \frac{k}{1+R} 
    = K - \frac{1}{1+R} \left( \sum_{\tau=1}^{T} \sum_{i=1}^N \sum_{k=1}^K p_{i\tau k} \cdot k \right).
\end{align*}
Substituting the definition $R = \frac{1}{K} \sum_{\tau, i, k} p_{i\tau k} \cdot k$, we obtain:
\[
\mathbb{E}[B_t] \geq K - \frac{1}{1+R} (K \cdot R) = K \left( 1 - \frac{R}{1+R} \right) = \frac{K}{1+R}.
\]
Dividing by $K$, we confirm that $\gamma_{tj} \geq \mathbb{E}[B_t]/K \geq \frac{1}{1+R}$. This completes the induction.

\medskip
Next, we will derive the RFE-FR guarantee.
With the condition $\gamma_{tj} \geq \frac{1}{1+R}$ established, let $q_{tj} = \frac{1}{(1+R)\gamma_{tj}}$ be the valid probability defined in the algorithm. Fix any triple $(i,t,j)$ with $p_{itj}>0$.

Conditioned on the arrival event $S_{itj}=1$, the algorithm uses an independent Bernoulli coin: with probability $q_{tj}$, it allocates $\min(B_t,j)$, and otherwise it allocates $0$. Since $B_t$ depends only on arrivals and allocations in slots $1,\ldots,t-1$, and arrivals are independent across time slots, $B_t$ is independent of the arrival event $S_{itj}$ at slot $t$. Using the linearity of expectation, we obtain
\[
\mathbb{E}[C_{it} \mid S_{itj} = 1]
= q_{tj} \cdot \mathbb{E}[\min(B_t, j)\mid S_{itj}=1].
\]
By the above independence,
\[
\mathbb{E}[C_{it} \mid S_{itj} = 1]
= q_{tj} \cdot \mathbb{E}[\min(B_t, j)].
\]
Substituting $q_{tj} = \frac{1}{(1+R)\gamma_{tj}}$:
\[
\mathbb{E}[C_{it} \mid S_{itj} = 1] = \frac{1}{(1+R)\gamma_{tj}} \cdot \mathbb{E}[\min(B_t, j)].
\]
Recall the definition $\gamma_{tj} := \mathbb{E}[\min(B_t/j, 1)]$. Since $j$ is a constant, we have the identity $\mathbb{E}[\min(B_t, j)] = j \cdot \gamma_{tj}$. Substituting:
\[
\mathbb{E}[C_{it} \mid S_{itj} = 1] = \frac{1}{(1+R)\gamma_{tj}} \cdot (j \cdot \gamma_{tj}) = \frac{j}{1+R}.
\]
Thus, the algorithm attains an RFE-FR of $\frac{1}{1+R}$.
\end{proof}

\begin{remark}
Algorithm~\ref{alg:non-iid} relies on the pre-computed values $\gamma_{tj}$ to calibrate acceptance probabilities. These values are computed via Algorithm~\ref{alg:dynamic-beta} in Appendix~\ref{sec:beta-calc} using a polynomial-time dynamic programming approach.
\end{remark}

\subsection{Extension to Heterogeneous Priorities}\label{sec:time-varying-het}

In many applications, enforcing identical service guarantees across all groups may be neither desirable nor equitable. For instance, in emergency medical supply distribution, high-risk populations often require prioritized access compared to the general public; in food bank operations, families with children might be prioritized over single adults. To accommodate such heterogeneous needs, we now extend our framework to handle arbitrary priority vectors $\beta = (\beta_1, \ldots, \beta_N) \in (0,1]^N$, normalized so that $\max_{i \in [N]} \beta_i = 1$.

The key insight is that priority weights can be incorporated through a reduction that filters the incoming request stream based on priority levels. Specifically, upon receiving a request from group $i$, we first perform a randomized screening step that accepts the request with probability $\beta_i$. Rejected requests receive zero allocation immediately. Accepted requests are then processed by our optimal unit-priority algorithm (Algorithm~\ref{alg:non-iid}), but tuned to the reduced \emph{virtual load} $R_\beta$ rather than the raw load $R$.

This two-stage procedure effectively transforms the original instance into a virtual instance where the arrival probability of a request from group $i$ with demand $j$ is reduced from $p_{itj}$ to $\beta_i \cdot p_{itj}$. The expected load on this virtual instance is precisely $R_\beta$, the priority-weighted load parameter defined in~\eqref{eq:R-beta}.

\begin{algorithm}[h]
\caption{Weighted FORA-IU Algorithm}\label{alg:weighted-fora}
\begin{algorithmic}[1]
\Require Parameters $K, N, T$, probabilities $p_{itj}$, and priority weights $\beta$.
\Ensure Resource allocation decision $C_{it}$ for each time slot.
\State \textbf{Virtual Instance Construction:} Define a virtual instance $\mathcal{I}'$ with arrival probabilities $p'_{itj} := \beta_i \cdot p_{itj}$.
\State Compute $\gamma'_{tj} := \mathbb{E}[\min(B'_t/j, 1)]$ for $\mathcal{I}'$ using Algorithm~\ref{alg:dynamic-beta}, where $R' = R_\beta$.
\State Initialize $B_1 := K$ and $C_{it} := 0$ for all $i,t$.
\For{$t = 1$ \textbf{to} $T$}
    \If{a request from group $i$ with demand $j$ arrives ($S_{itj}=1$)}
        \State \textbf{Priority-Based Screening:} Draw $Z_t \sim \text{Bernoulli}(\beta_i)$.
        \If{$Z_t = 1$}
\State Allocate $C_{it}=\min(B_t,j)$ with probability 
$q_{tj}=\frac{1}{(1+R_\beta)\gamma'_{tj}}$.
\State Otherwise, allocate $C_{it}=0$.
        \Else
            \State Allocate $C_{it} = 0$.
        \EndIf
    \EndIf
    \State Set $B_{t+1} := B_t - \sum_{i=1}^N C_{it}$.
\EndFor
\end{algorithmic}
\end{algorithm}

\begin{theorem}\label{thm:weighted-fora}
Algorithm~\ref{alg:weighted-fora} achieves an FE-FR-$\beta$ of $\frac{1}{1+R_\beta}$ for any FORA-IU instance with time-varying arrivals.
\end{theorem}

\begin{proof}
The algorithm filters incoming requests: when a request from group $i$ arrives, it is passed to the allocation mechanism only with probability $\beta_i$.

This process creates a new stream where the probability of a request from group $i$ with demand $j$ is reduced from $p_{itj}$ to $p'_{itj} = \beta_i p_{itj}$. The total expected load the algorithm faces is therefore:
\[
R' = \frac{1}{K}\sum_{t, i, j} p'_{itj} \cdot j = \frac{1}{K}\sum_{t, i, j} \beta_i p_{itj} \cdot j = R_\beta.
\]
Conditional on the screening outcomes, the sequence of forwarded requests is distributionally identical to the arrival process of the virtual instance $\mathcal{I}'$, and the allocation rule applied to forwarded requests coincides with Algorithm~\ref{alg:non-iid} on $\mathcal{I}'$. Consequently, the actual residual budget $B_t$ used by Algorithm~\ref{alg:weighted-fora} has the same distribution as the virtual residual budget $B'_t$ used to define $\gamma'_{tj}$. Since the core allocation logic (Algorithm~\ref{alg:non-iid}) is applied to this virtual stream with parameter $R' = R_\beta$, it guarantees that every accepted request of size $j$ receives an expected allocation of $\frac{j}{1+R_\beta}$.

Now consider the perspective of the original group $i$. A request is accepted with probability $\beta_i$. Therefore, the unconditional expected allocation for any arrival of size $j$ is:
\[
\mathbb{E}[\text{Allocation}] = \mathbb{P}(\text{Accepted}) \times \mathbb{E}[\text{Allocation} \mid \text{Accepted}] = \beta_i \times \frac{j}{1+R_\beta}.
\]
Summing over all time slots, the total expected allocation for group $i$ is exactly $\frac{\beta_i}{1+R_\beta}$ times its total expected demand. This confirms that the policy achieves an FE-FR-$\beta$ of $\frac{1}{1+R_\beta}$.
\end{proof}

\subsection{Optimality of the Guarantee}

We now establish that the FE-FR-$\beta$ of $\frac{1}{1+R_\beta}$ achieved by Algorithm~\ref{alg:weighted-fora} is the best possible for any online policy in the general time-varying setting.

\begin{example}\label{exp:tight-non-iid-beta-active}
Fix an integer $N \ge 2$ and a priority vector $\beta = (\beta_1, \ldots, \beta_N) \in (0,1]^N$ with $\max_{i \in [N]} \beta_i = 1$. Without loss of generality, relabel so that $\beta_N = 1$. Fix a target load value $\rho > 0$ and a parameter $\varepsilon \in (0, \min\{1, \rho\})$.

Choose a horizon length $T \ge 2$ large enough so that
\begin{equation}\label{eq:general-prob-valid}
(N-1) \cdot \frac{\rho - \varepsilon}{(T-1)\sum_{m=1}^{N-1}\beta_m} \le 1,
\end{equation}
which ensures that the per-slot arrival probability is valid. Assume there are $K$ indivisible units of resources.

Define an instance in which every arrival (when it occurs) demands exactly $K$ units:
\[
p_{itK} = 
\begin{cases}
\displaystyle \frac{\rho - \varepsilon}{(T-1)\sum_{m=1}^{N-1}\beta_m} & \text{if } i \in \{1, \ldots, N-1\},\ t \in \{1, \ldots, T-1\}, \\[1.5ex]
\varepsilon & \text{if } i = N,\ t = T, \\
0 & \text{otherwise}.
\end{cases}
\]

The instance parameter $R_\beta$ equals $\rho$:
\[
R_\beta = \frac{1}{K}\left(\sum_{t=1}^{T-1}\sum_{i=1}^{N-1}\beta_i p_{itK} K + \beta_N p_{NTK} K\right) = (\rho - \varepsilon) + \varepsilon = \rho.
\]
\end{example}

This construction creates a dilemma. During the first $T-1$ slots, requests arrive only from lower-priority groups. If a policy serves these fairly, it uses up its entire budget, leaving insufficient resources for the high-priority group $N$ that might arrive in the final slot. Conversely, saving resources for that potential late arrival means treating the earlier groups unfairly.

\begin{theorem}\label{thm:optimal-rFE-FR-beta-active}
Consider the general FORA-IU problem with $N \ge 2$ groups and priority coefficients $\beta$. No online policy can guarantee an FE-FR-$\beta$ strictly larger than $\frac{1}{1+R_\beta}$ on all instances.
\end{theorem}

\begin{proof}
Fix $\rho > 0$ and consider the instances from Example~\ref{exp:tight-non-iid-beta-active} for an arbitrary (small) $\varepsilon \in (0, \min\{1, \rho\})$. Let $\pi$ be any online policy and let $\alpha$ be the FE-FR-$\beta$ it achieves.

Let
\[
X_{\mathrm{pre}} := \sum_{t=1}^{T-1}\sum_{i=1}^{N-1} C_{it} \qquad \text{and} \qquad B_T := K - X_{\mathrm{pre}}
\]
denote, respectively, the total allocation made before time $T$ and the remaining budget at the start of slot $T$.

For each group $i \in \{1, \ldots, N-1\}$, all demand occurs in slots $1, \ldots, T-1$, and every arrival demands $K$ units. Hence,
\[
\mathbb{E}\left[\sum_{t=1}^{T-1} C_{it}\right] \ge \alpha \beta_i \mathbb{E}\left[\sum_{t=1}^{T-1} S_{itK} \cdot K\right] = \alpha \beta_i \sum_{t=1}^{T-1} p_{itK} K.
\]
Summing over $i = 1, \ldots, N-1$ gives
\[
\mathbb{E}[X_{\mathrm{pre}}] \ge \alpha K \sum_{t=1}^{T-1}\sum_{i=1}^{N-1} \beta_i p_{itK} = \alpha K(\rho - \varepsilon),
\]
where the last equality follows from the construction of $p_{itK}$.

Therefore,
\begin{equation}\label{eq:BT-upper}
\mathbb{E}[B_T] = K - \mathbb{E}[X_{\mathrm{pre}}] \le K(1 - \alpha(\rho - \varepsilon)).
\end{equation}

At time $T$, only group $N$ can arrive, with probability $\varepsilon$, and its demand is $K$. Thus $C_{NT} \le B_T S_{NTK}$, which implies
\[
\mathbb{E}[C_{NT}] \le \mathbb{E}[B_T S_{NTK}] = \varepsilon \mathbb{E}[B_T],
\]
using independence across time slots.

On the other hand, the FE-FR-$\beta$ constraint for group $N$ requires
\[
\mathbb{E}[C_{NT}] \ge \alpha \beta_N \mathbb{E}[S_{NTK} \cdot K] = \alpha \varepsilon K,
\]
since $\beta_N = 1$. Combining the bounds yields
\[
\alpha \varepsilon K \le \varepsilon \mathbb{E}[B_T] \le \varepsilon K(1 - \alpha(\rho - \varepsilon)).
\]
Dividing by $\varepsilon K > 0$ gives
\[
\alpha \le 1 - \alpha(\rho - \varepsilon) \quad \Longrightarrow \quad \alpha(1 + \rho - \varepsilon) \le 1 \quad \Longrightarrow \quad \alpha \le \frac{1}{1 + \rho - \varepsilon}.
\]

Finally, recall that $R_\beta = \rho$ in Example~\ref{exp:tight-non-iid-beta-active}. Let $\delta > 0$ be arbitrary. Since $\lim_{\varepsilon \downarrow 0} \frac{1}{1+\rho-\varepsilon} = \frac{1}{1+\rho}$, choose $\varepsilon \in (0, \min\{1, \rho\})$ small enough so that
\[
\frac{1}{1+\rho-\varepsilon} < \frac{1}{1+\rho} + \delta.
\]
The bound then implies $\alpha \le \frac{1}{1+\rho} + \delta$ on an instance with $R_\beta = \rho$. As $\delta$ is arbitrary, no policy can guarantee an FE-FR-$\beta$ strictly larger than $\frac{1}{1+\rho} = \frac{1}{1+R_\beta}$ uniformly over all instances.
\end{proof}

\section{Optimal Online Policy for Time-Invariant Request Distributions}\label{sec:time-invariant}

In this section, we restrict attention to FORA-IU instances in which the arrival distribution does not vary over time. Formally, we assume that for each group $i \in [N]$ and demand size $j \in [K]$, there exist probabilities $p_{ij} \ge 0$ such that
\[
p_{itj} = p_{ij} \qquad \text{for all } t \in [T],\ i \in [N],\ j \in [K],
\]
with $\sum_{i=1}^N \sum_{j=1}^K p_{ij} \le 1$ to ensure the probability of no arrival in any slot is nonnegative. Our goal is to design an online policy that maximizes the worst-case FE-FR-$\beta$ in this restricted class, which is expected to be strictly better than $\frac{1}{1+R_\beta}$.

We introduce notation that will be used throughout this section. For each group $i$, the per-slot expected demand is
\(
d_i := \sum_{j=1}^K p_{ij} \cdot j,
\)
and the priority-weighted load parameter $R_\beta$ defined in~\eqref{eq:R-beta} simplifies to
\begin{equation}\label{eq:R-beta-time-invariant}
R_\beta = \frac{1}{K}\sum_{t=1}^T \sum_{i=1}^N \sum_{j=1}^K \beta_i p_{ij} j = \frac{T}{K}\sum_{i=1}^N \beta_i d_i.
\end{equation}
We also define the \emph{normalized per-slot priority-weighted demand}
\begin{equation}\label{eq:W-beta-def}
W_\beta := \frac{R_\beta}{T} = \frac{1}{K}\sum_{i=1}^N \beta_i d_i = \frac{1}{K}\sum_{i=1}^N \sum_{j=1}^K \beta_i p_{ij} j.
\end{equation}
Since at most one request arrives per slot, each demand satisfies $j \le K$, and $\beta_i \le 1$ for all $i$, we have
\[
\sum_{i=1}^N \sum_{j=1}^K \beta_i p_{ij} \cdot j \le \sum_{i=1}^N \sum_{j=1}^K p_{ij} \cdot j = \mathbb{E}[\text{demand in a slot}] \le K,
\]
which implies $W_\beta \in [0,1]$ and $R_\beta = T W_\beta \in [0,T]$.

\subsection{The Homogeneous-Priority Case}

We begin by analyzing the homogeneous-priority setting where $\beta_i = 1$ for all groups $i \in [N]$. In this case, the priority-weighted load $R_\beta$ reduces to the standard load parameter
\begin{equation}\label{eq:R-time-invariant}
R = \frac{T}{K}\sum_{i=1}^N d_i,
\end{equation}
and the normalized per-slot demand becomes
\begin{equation}\label{eq:W-def}
W := \frac{R}{T} = \frac{1}{K}\sum_{i=1}^N d_i = \frac{1}{K}\sum_{i=1}^N\sum_{j=1}^K p_{ij} j.
\end{equation}

We introduce a simple online policy called \emph{Random Cyclic Blocks} (RCB) and analyze its FE-FR guarantee. The algorithm views the $K$ resource units as positions arranged on a circle. When a request with demand $j$ arrives, the algorithm randomly selects a starting position on this circle and attempts to claim a contiguous block of $j$ positions. Units in the block that are still available are allocated to the request; units already consumed are skipped (see Algorithm~\ref{alg:rcb}).

\begin{algorithm}[h]
\caption{Random Cyclic Blocks (RCB)}\label{alg:rcb}
\begin{algorithmic}[1]
\Require Total units $K$, horizon $T$, and time-invariant arrival probabilities $p_{ij}$.
\Ensure Allocations $C_{it}$ for all $i \in [N]$, $t \in [T]$.
\State \textbf{Initialize:} For each unit $u \in \{1,\dots,K\}$, set \texttt{free}$(u) \leftarrow \texttt{true}$.
\For{each slot $t = 1,2,\dots,T$}
    \State Observe the arrival for time slot $t$.
    \If{a request arrives from group $i$ with demand size $j$}
        \State Draw a start index $S_t$ uniformly at random from $\{1, \dots, K\}$.
        \State Define the cyclic target block of size $j$:
        \[ H_t := \{S_t, S_t+1, \dots, S_t+j-1\} \pmod K \]
        \State Identify available units in the block:
        \[ Q_t := \{u \in H_t : \texttt{free}(u)=\texttt{true}\} \]
        \State Allocate $C_{it} \leftarrow \lvert Q_t \rvert$ units to this request.
        \State Update \texttt{free}$(u) \leftarrow \texttt{false}$ for all $u \in Q_t$.
    \Else
        \State $C_{it} \leftarrow 0$ for all $i$.
    \EndIf
\EndFor
\end{algorithmic}
\end{algorithm}

The algorithm is straightforward to verify as a valid policy. Since $H_t$ contains exactly $j$ positions, the allocation $\lvert Q_t \rvert$ is an integer between $0$ and $j$. Furthermore, each unit $u$ transitions from available to consumed at most once, so the total allocation never exceeds $K$.

To analyze the expected filling ratio, we exploit the rotational symmetry of the circular structure. Rather than tracking the complex interactions among overlapping blocks, we focus on the probability that a single fixed unit $u$ is consumed by a particular group.

\begin{lemma}\label{lem:unit-prob}
Fix an arbitrary resource unit $u \in \{1, \dots, K\}$. In any given time slot, the probability that unit $u$ falls within the target block of a request from group $i$ is
\[
q_i := \frac{1}{K} \sum_{j=1}^K p_{ij} \cdot j.
\]
Consequently, the probability that unit $u$ is requested by any group in a given time slot equals $W := \sum_{i=1}^N q_i = \frac{R}{T}$.
\end{lemma}

\begin{proof}
Suppose a request from group $i$ arrives with demand $j$. The algorithm selects a start index $S_t$ uniformly at random from $\{1, \dots, K\}$. Since the request occupies a contiguous block of $j$ positions (with wraparound), exactly $j$ starting positions result in the block covering the fixed unit $u$. Thus, conditioned on a request of size $j$ arriving, unit $u$ is requested with probability $j/K$.

The unconditional probability that group $i$ requests unit $u$ is obtained by summing over all possible demand sizes:
\[
q_i = \sum_{j=1}^K \mathbb{P}(\text{Group } i \text{ arrives with demand } j) \cdot \frac{j}{K} = \sum_{j=1}^K p_{ij} \cdot \frac{j}{K} = \frac{1}{K} \sum_{j=1}^K p_{ij} \cdot j.
\]
Since at most one group arrives per time slot, the arrival events are mutually exclusive. Summing over all groups yields the total probability that unit $u$ is requested:
\[
W = \sum_{i=1}^N q_i = \frac{1}{K} \sum_{i=1}^N \sum_{j=1}^K p_{ij} \cdot j.
\]
Comparing with~\eqref{eq:R-time-invariant}, we obtain $W = R/T$.
\end{proof}

Under RCB, a unit is allocated to the first request whose block covers it. This observation enables us to compute the expected total allocation via a geometric series.

\begin{lemma}\label{lem:unit-alloc}
Fix a specific unit $u \in \{1, \dots, K\}$. Under RCB, the probability that unit $u$ is allocated to group $i$ over the horizon (equivalently, the expectation of the corresponding indicator) is
\[
\mathbb{P}(\text{unit } u \text{ is allocated to group } i)
= q_i \cdot \frac{1 - (1 - W)^T}{W}.
\]
\end{lemma}

\begin{proof}
For each $t \in [T]$, let $A_t$ denote the event that unit $u$ is allocated to group $i$ at time $t$.
Event $A_t$ occurs if and only if: (1) group $i$ requests unit $u$ at time $t$, and (2) no group requested unit $u$ in any prior slot $1, \dots, t-1$.
By Lemma~\ref{lem:unit-prob}, the probability of (2) is $(1-W)^{t-1}$ and the probability of (1) is $q_i$.
By independence across time slots, $\mathbb{P}(A_t) = (1-W)^{t-1} \cdot q_i$.

Since unit $u$ can be allocated at most once, the events $\{A_t\}_{t=1}^T$ are disjoint. Therefore,
\begin{align*}
\mathbb{P}(\text{unit } u \text{ is allocated to group } i)
&= \sum_{t=1}^T \mathbb{P}(A_t)
= \sum_{t=1}^T (1 - W)^{t-1} \cdot q_i \\
&= q_i \sum_{k=0}^{T-1} (1-W)^k
= q_i \cdot \frac{1 - (1 - W)^T}{W}.
\end{align*}
\end{proof}

Aggregating this single-unit analysis yields the FE-FR guarantee for the algorithm.

\begin{theorem}\label{thm:rcb-fe-fr}
Consider an unweighted time-invariant FORA-IU instance with load parameter $R$. Algorithm~\ref{alg:rcb} achieves
\[
\mathrm{FE\text{-}FR}(\mathrm{RCB}) = \frac{1 - (1 - R/T)^T}{R},
\]
with the convention that this expression equals $1$ when $R = 0$. Furthermore, this guarantee is lower-bounded by $\frac{1 - e^{-R}}{R}$ (also interpreted as $1$ at $R=0$). 
\end{theorem}

\begin{proof}
The total expected allocation to group $i$ equals the sum of expected allocations across all $K$ units. By symmetry, Lemma~\ref{lem:unit-alloc} applies identically to each unit:
\[
\mathbb{E}\left[\sum_{t=1}^T C_{it}\right] = K \cdot \mathbb{E}[\text{allocation of } u \text{ to group } i] = K \cdot q_i \cdot \frac{1 - (1 - W)^T}{W}.
\]
Substituting $q_i = \frac{1}{K}\sum_{j} p_{ij} j$ and $W = R/T$:
\[
\mathbb{E}\left[\sum_{t=1}^T C_{it}\right] = \left(\sum_{j=1}^K p_{ij} j\right) \cdot \frac{1 - (1 - R/T)^T}{R/T}.
\]
The expected total demand for group $i$ is $\mathbb{E}[\mathrm{Demand}_i] = T \sum_{j=1}^K p_{ij} j$. For any group with positive expected demand, dividing the expected allocation by the expected demand:
\[
\mathrm{FE\text{-}FR} = \frac{\left(\sum_{j} p_{ij} j\right) \cdot \frac{1 - (1 - R/T)^T}{R/T}}{T \left(\sum_{j} p_{ij} j\right)} = \frac{1 - (1 - R/T)^T}{R}.
\]
The inequality $(1 - x/T)^T \le e^{-x}$ for $x > 0$ and $T \ge 1$ implies $1 - (1 - R/T)^T \ge 1 - e^{-R}$, yielding the lower bound $\frac{1 - e^{-R}}{R}$.
\end{proof}

Since $\frac{1-e^{-R}}{R} > \frac{1}{1+R}$ for all $R > 0$, Theorem~\ref{thm:rcb-fe-fr} demonstrates that RCB strictly improves upon the guarantee $\frac{1}{1+R}$ achieved by Algorithm~\ref{alg:non-iid}. This improvement quantifies the value of distributional stationarity, demonstrating that stable arrival patterns allow the algorithm to exploit the symmetry of the circular allocation structure for better fairness. This advantage further holds for heterogeneous-priority scenarios.

\subsection{Extension to Heterogeneous Priorities}

We now extend the RCB algorithm to handle heterogeneous priority weights $\beta = (\beta_1, \ldots, \beta_N) \in (0,1]^N$. The extension follows the same reduction strategy developed in Section~\ref{sec:time-varying-het} where we incorporate priority weights by screening incoming requests based on their priority levels before passing them to the core allocation mechanism.

Specifically, when a request from group $i$ arrives, we first perform a randomized screening step that accepts the request with probability $\beta_i$. Rejected requests receive zero allocation immediately. Accepted requests are then processed by the RCB algorithm, which now operates on a reduced effective load. This filtering reduces the per-slot probability that any group $i$ contributes to the system from $p_{ij}$ to $\beta_i p_{ij}$, yielding an effective priority-weighted load of $R_\beta$ rather than $R$.

\begin{algorithm}[h]
\caption{Weighted Random Cyclic Blocks}\label{alg:weighted-rcb}
\begin{algorithmic}[1]
\Require Total units $K$, horizon $T$, time-invariant probabilities $p_{ij}$, and priority weights $\beta$.
\Ensure Allocations $C_{it}$ for all $i \in [N]$, $t \in [T]$.
\State \textbf{Initialize:} For each unit $u \in \{1,\dots,K\}$, set \texttt{free}$(u) \leftarrow \texttt{true}$.
\For{$t = 1$ \textbf{to} $T$}
    \State Observe the arrival for time slot $t$.
    \If{a request arrives from group $i$ with demand $j$}
        \State \textbf{Priority-Based Screening:} Draw $Z_t \sim \mathrm{Bernoulli}(\beta_i)$.
        \If{$Z_t = 0$}
            \State Reject the request: $C_{it} \gets 0$.
        \Else
            \State \textbf{RCB Allocation:}
            \State Draw start index $S_t \sim \mathrm{Uniform}(\{1, \dots, K\})$.
            \State Define target block: $H_t := \{S_t, S_t+1, \dots, S_t+j-1\} \pmod K$.
            \State Identify available units: $Q_t := \{u \in H_t : \texttt{free}(u)=\texttt{true}\}$.
            \State Allocate $C_{it} \leftarrow \lvert Q_t \rvert$.
            \State Update \texttt{free}$(u) \leftarrow \texttt{false}$ for all $u \in Q_t$.
        \EndIf
    \Else
        \State $C_{it} \leftarrow 0$ for all $i$.
    \EndIf
\EndFor
\end{algorithmic}
\end{algorithm}

\begin{theorem}\label{thm:weighted-RCB}
Algorithm~\ref{alg:weighted-rcb} achieves an FE-FR-$\beta$ of
\[
\frac{1-(1-R_\beta/T)^T}{R_\beta},
\]
with the convention that the expression equals $1$ when $R_\beta=0$. In particular, it is at least $\frac{1-e^{-R_\beta}}{R_\beta}$, again interpreted as $1$ at $R_\beta=0$.
\end{theorem}

\begin{proof}
If $R_\beta=0$, then $d_i=0$ for every group $i$, so the FE-FR-$\beta$ claim is trivial. Assume henceforth that $R_\beta>0$.

After the Bernoulli screening step, the algorithm operates on a virtual time-invariant instance with arrival probabilities
\(
p'_{ij} := \beta_i p_{ij}.
\)
For this virtual instance, define
\[
d_i' := \sum_{j=1}^K p'_{ij}j = \beta_i d_i,
\qquad
q_i' := \frac{d_i'}{K} = \frac{\beta_i d_i}{K},
\qquad
W' := \sum_{i=1}^N q_i' = \frac{1}{K}\sum_{i=1}^N \beta_i d_i = \frac{R_\beta}{T}.
\]

Applying Lemma~\ref{lem:unit-alloc} to this virtual instance shows that for any fixed unit $u$,
\[
\mathbb{P}(\text{unit }u\text{ is allocated to group }i)
=
q_i' \cdot \frac{1-(1-W')^T}{W'}.
\]
Summing over all $K$ units yields
\[
\mathbb{E}\left[\sum_{t=1}^T C_{it}\right]
=
K q_i' \cdot \frac{1-(1-W')^T}{W'}
=
\beta_i d_i \cdot \frac{1-(1-W')^T}{W'}.
\]

Since the expected total demand of group $i$ in the original instance is
\[
\mathbb{E}\left[\sum_{t=1}^T\sum_{j=1}^K S_{itj}j\right] = T d_i,
\]
we obtain
\[
\mathbb{E}\left[\sum_{t=1}^T C_{it}\right]
=
\beta_i \,
\mathbb{E}\left[\sum_{t=1}^T\sum_{j=1}^K S_{itj}j\right]
\cdot
\frac{1-(1-W')^T}{T W'}.
\]
Using $W' = R_\beta/T$, this becomes
\[
\mathbb{E}\left[\sum_{t=1}^T C_{it}\right]
=
\beta_i \,
\mathbb{E}\left[\sum_{t=1}^T\sum_{j=1}^K S_{itj}j\right]
\cdot
\frac{1-(1-R_\beta/T)^T}{R_\beta}.
\]
Hence the algorithm achieves FE-FR-$\beta$ of $\frac{1-(1-R_\beta/T)^T}{R_\beta}$.

Finally, since $(1-x/T)^T \le e^{-x}$ for $x \in [0,T]$, we have
\[
\frac{1-(1-R_\beta/T)^T}{R_\beta} \ge \frac{1-e^{-R_\beta}}{R_\beta}.
\]
This completes the proof.
\end{proof}

In particular, Algorithm~\ref{alg:weighted-rcb} guarantees FE-FR-$\beta$ of at least $\frac{1-e^{-R_\beta}}{R_\beta}$. Since $\frac{1-e^{-R_\beta}}{R_\beta} > \frac{1}{1+R_\beta}$ for every $R_\beta>0$, time-invariant arrivals admit a strictly stronger load-dependent guarantee than the optimal universal bound for arbitrary time-varying arrivals.

While the priority-based screening approach provides a conceptually simple reduction to the unweighted setting, it is not merely a convenient heuristic. One might ask whether a more sophisticated, state-dependent prioritization scheme could achieve better fairness guarantees. In the following subsection, we answer this question in the negative by proving that the bound achieved by Algorithm~\ref{alg:weighted-rcb} is tight. In particular, no online policy can guarantee an FE-FR-$\beta$ strictly larger than the exact finite-horizon expression $\frac{1 - (1 - R_\beta/T)^T}{R_\beta}$ for all time-invariant instances.

\subsection{Optimality of the Guarantee}

We now establish that the FE-FR-$\beta$ of $\frac{1-(1-R_\beta/T)^T}{R_\beta}$ achieved by Algorithm~\ref{alg:weighted-rcb} is the best possible for any online policy in the time-invariant setting.

\begin{example}[Full-support $\varepsilon$-hard family]\label{exp:full-support-eps-hard}
Fix $\beta\in(0,1]^N$ with $\max_{i\in[N]}\beta_i=1$, and fix $\rho\in(0,T)$.
Choose $k\in[N]$ such that $\beta_k=1$, and define
\[
\mathcal{A}:=[N]\times[K]\setminus\{(k,K)\}.
\]
Let
\[
A:=\sum_{(i,j)\in\mathcal{A}}\frac{\beta_i j}{K},
\qquad
\Gamma:=\sum_{(i,j)\in\mathcal{A}}\left(1-\frac{\beta_i j}{K}\right).
\]
For a parameter $\lambda>0$, define
\[
p_{ij}^{(\lambda)}=
\begin{cases}
\dfrac{\rho-\lambda A}{T}, & \text{if } (i,j)=(k,K),\\[1.5ex]
\dfrac{\lambda}{T}, & \text{if } (i,j)\in\mathcal{A}.
\end{cases}
\]
and set $p_{itj}^{(\lambda)} := p_{ij}^{(\lambda)}$ for all $t\in[T]$.
If $\lambda>0$ is small enough so that
\[
\rho-\lambda A>0
\qquad\text{and}\qquad
\rho+\lambda\Gamma\le T,
\]
then this is a feasible full-support time-invariant instance, i.e.,
\[
p_{ij}^{(\lambda)}>0
\qquad\text{for all }i\in[N],\ j\in[K].
\]

Moreover,
\[
R_\beta
=
\frac{1}{K}\sum_{t=1}^T\sum_{i=1}^N\sum_{j=1}^K \beta_i p_{itj}^{(\lambda)} j
=
\rho.
\]
Indeed,
\[
R_\beta
=
(\rho-\lambda A)
+
\lambda\sum_{(i,j)\in\mathcal{A}}\frac{\beta_i j}{K}
=
(\rho-\lambda A)+\lambda A
=
\rho.
\]

Let
\[
q_\lambda:=\sum_{i=1}^N\sum_{j=1}^K p_{ij}^{(\lambda)}
=
\frac{\rho-\lambda A}{T}+\frac{\lambda\lvert\mathcal{A}\rvert}{T}
=
\frac{\rho+\lambda\Gamma}{T}.
\]
Now fix any online policy, and let $\alpha$ denote the FE-FR-$\beta$ it achieves on this instance. Summing the FE-FR-$\beta$ inequalities over all groups gives
\[
\sum_{i=1}^N \E\!\left[\sum_{t=1}^T C_{it}\right]
\ge
\alpha \sum_{i=1}^N \beta_i \E\!\left[\sum_{t=1}^T\sum_{j=1}^K S_{itj}j\right]
=
\alpha \rho K.
\]
On the other hand, the total allocation over the horizon is at most $K$, and it is zero if no request arrives at all. Therefore,
\[
\sum_{i=1}^N \E\!\left[\sum_{t=1}^T C_{it}\right]
=
\E\!\left[\sum_{i=1}^N\sum_{t=1}^T C_{it}\right]
\le
K\,\Pr(\text{at least one request arrives during the horizon}).
\]
Since arrivals are independent across time and the per-slot arrival probability is $q_\lambda$,
\[
\Pr(\text{at least one request arrives during the horizon})
=
1-(1-q_\lambda)^T.
\]
Hence any online policy achieving FE-FR-$\beta$ of $\alpha$ on this instance must satisfy
\[
\alpha
\le
\frac{1-(1-q_\lambda)^T}{\rho}
=
\frac{1-\left(1-\frac{\rho+\lambda\Gamma}{T}\right)^T}{\rho}.
\]
Since the function
\[
g(x):=1-\left(1-\frac{x}{T}\right)^T
\]
is $1$-Lipschitz on $[0,T]$, it follows that
\[
\alpha
\le
\frac{1-(1-\rho/T)^T}{\rho}
+\frac{\lambda\Gamma}{\rho}.
\]
Therefore, for every $\varepsilon>0$, choosing $\lambda>0$ small enough that
\(
\lambda\Gamma\le \rho\varepsilon
\)
yields a full-support time-invariant instance with $R_\beta=\rho$ such that
every online policy satisfies
\[
\alpha
\le
\frac{1-(1-\rho/T)^T}{\rho}
+\varepsilon.
\]
\end{example}

\begin{theorem}\label{thm:eps-upper-no-b}
Fix $\beta\in(0,1]^N$ with $\max_{i\in[N]}\beta_i=1$, and fix $\rho\in(0,T)$.
Let $\mathcal I_{\mathrm{fs}}(\rho,\beta)$ denote the class of all
time-invariant FORA-IU instances satisfying
\[
R_\beta=\rho
\qquad\text{and}\qquad
p_{ij}>0 \ \text{ for all } i\in[N],\ j\in[K].
\]
Then no online policy can guarantee an FE-FR-$\beta$ strictly larger than
\[
\frac{1-(1-\rho/T)^T}{\rho}
\]
on every instance in $\mathcal I_{\mathrm{fs}}(\rho,\beta)$.

Consequently, the optimal worst-case FE-FR-$\beta$ guarantee over
$\mathcal I_{\mathrm{fs}}(\rho,\beta)$ is exactly
\[
\frac{1-(1-\rho/T)^T}{\rho}
=
\frac{1-(1-R_\beta/T)^T}{R_\beta}.
\]
\end{theorem}

\begin{proof}
Let
\[
L(\rho,T):=\frac{1-(1-\rho/T)^T}{\rho}.
\]

First, by Theorem~\ref{thm:weighted-RCB}, Algorithm~\ref{alg:weighted-rcb}
achieves FE-FR-$\beta$ of at least $L(\rho,T)$ for every time-invariant instance.
In particular, it achieves this guarantee for every instance in
$\mathcal I_{\mathrm{fs}}(\rho,\beta)$.

It remains to prove that no online policy can guarantee strictly more than
$L(\rho,T)$ on the whole class $\mathcal I_{\mathrm{fs}}(\rho,\beta)$.
Suppose otherwise. Then there exist an online policy $\pi$ and a constant
$\delta>0$ such that $\pi$ guarantees FE-FR-$\beta$ of at least
\(
L(\rho,T)+\delta
\)
on every instance in $\mathcal I_{\mathrm{fs}}(\rho,\beta)$.

Now apply Example~\ref{exp:full-support-eps-hard} with
\(
\varepsilon:=\delta/2.
\)
It yields an instance in $\mathcal I_{\mathrm{fs}}(\rho,\beta)$ with
$R_\beta=\rho$ such that every online policy, and hence in particular $\pi$,
must satisfy
\[
\alpha
\le
L(\rho,T)+\frac{\delta}{2}
<
L(\rho,T)+\delta.
\]
This contradicts the assumed guarantee of $\pi$.

Therefore, no online policy can guarantee an FE-FR-$\beta$ strictly larger than
$L(\rho,T)$ on every instance in $\mathcal I_{\mathrm{fs}}(\rho,\beta)$.
Combining this upper bound with the lower bound from
Theorem~\ref{thm:weighted-RCB} proves that the optimal worst-case guarantee over
$\mathcal I_{\mathrm{fs}}(\rho,\beta)$ is exactly $L(\rho,T)$.
\end{proof}

The instances in $\mathcal I_{\mathrm{fs}}(\rho,\beta)$ form a subset of the time-invariant instances with $R_\beta=\rho$, so an upper bound that holds on this subset also holds on the full class.

Theorems~\ref{thm:weighted-RCB} and~\ref{thm:eps-upper-no-b} together establish that the Weighted RCB algorithm is optimal for the class of time-invariant request distributions. The exact achievable guarantee $\frac{1 - (1 - R_\beta/T)^T}{R_\beta}$, which is always at least $\frac{1 - e^{-R_\beta}}{R_\beta}$, strictly exceeds the bound $\frac{1}{1 + R_\beta}$ that is optimal in the general time-varying setting, quantifying the value of distributional stationarity in online fair allocation.

\section{Partial Allocation versus All-or-Nothing Allocation}\label{sec:all-or-nothing}

Throughout this paper, our algorithms employ a \emph{partial allocation} rule, wherein a request may receive any integer number of units up to its demand. An alternative paradigm, widely adopted in practice, is the \emph{all-or-nothing allocation} rule: each request must either be fulfilled completely or rejected entirely—no intermediate allocations are permitted.

A natural question arises: Can online policies restricted to the all-or-nothing rule \cite{alaei2012online,alaei2013online,yang2021competitive} achieve the same optimal FE-FR-$\beta$ guarantees established by our partial allocation algorithms? Specifically, can such policies attain an FE-FR of $\frac{1}{1+R}$ in the general time-varying setting or $\frac{1-(1-R/T)^T}{R}$ in the time-invariant setting? In what follows we establish a strict separation already against the weaker lower bound $\frac{1-e^{-R}}{R}\le \frac{1-(1-R/T)^T}{R}$, which suffices to rule out matching the exact optimal guarantee.

The answer is negative, even in the unweighted case where all priority coefficients are uniform. The following theorem formalizes this separation, demonstrating that partial allocation is essential for achieving optimal performance guarantees.

\begin{theorem}\label{thm:limit-all-or-nothing}
With uniform priority coefficients (i.e., $\beta_i=1$ for all $i\in [N]$), there exist instances where online policies restricted to the all-or-nothing allocation rule strictly fail to achieve the FE-FR guarantees of $\frac{1}{1+R}$ in the general case and $\frac{1-(1-R/T)^T}{R}$ in the time-invariant case, which are attainable via partial allocation. The separation in fact holds already against the weaker lower bound $\frac{1-e^{-R}}{R}$ in the time-invariant case.
\end{theorem}

The proof of Theorem~\ref{thm:limit-all-or-nothing} follows from the explicit constructions presented in the subsequent subsections: Theorem~\ref{thm:all-or-nothing-limit-general} establishes the separation for the general time-varying setting, and Theorem~\ref{thm:all-or-nothing-limit-iid} establishes it for the time-invariant setting.

\subsection{Impossibility in the General Case}

We construct instances in which all-or-nothing policies face an inherent capacity bottleneck that partial allocation circumvents. The key feature of this construction is that each request demands strictly more than half of the total capacity, ensuring that at most one request can ever be fully satisfied under the all-or-nothing rule.

\begin{example}\label{ex:all-or-nothing-general-T}
Let $K = \frac{1}{\epsilon}$ be the total available resource units, where $\epsilon > 0$ is chosen such that $K$ is an even integer. Consider an instance with $T$ time slots and $T$ groups.
In each time slot $t \in [T]$, a request from group $t$ arrives with probability $1$, demanding exactly $d = \frac{K}{2} + 1 = \frac{1}{2\epsilon} + 1$ units.

\end{example}

The following theorem quantifies the separation between all-or-nothing and partial allocation policies for this instance.

\begin{theorem}\label{thm:all-or-nothing-limit-general}
In the scenario described in Example~\ref{ex:all-or-nothing-general-T}, the maximum achievable FE-FR using any all-or-nothing allocation policy is exactly $\frac{1}{T}$. In contrast, Algorithm~\ref{alg:non-iid}, which uses partial allocation, achieves an FE-FR of $\frac{1}{1+R}$. If $\epsilon \leq \frac{1}{6}$ and $T>3$, then
\(
\frac{1}{T} < \frac{1}{1+R}.
\)
This demonstrates that all-or-nothing policies strictly underperform compared to partial-allocation policies in the general case.
\end{theorem}

\begin{proof}
Under the all-or-nothing rule, a request is either fully served in amount
\[
d=\frac{K}{2}+1=\frac{1}{2\epsilon}+1
\]
or rejected. Since
\[
2d = K+2 > K,
\]
at most one request can be fully served over the entire horizon.

For each group $i\in[T]$, let $Y_i$ be the indicator that group $i$'s request is fully served. Then
\(
\sum_{i=1}^T Y_i \le 1
\) almost surely,
and therefore
\(
\sum_{i=1}^T \mathbb{E}[Y_i] \le 1.
\)

Group $i$ has deterministic demand $d$ and receives allocation $dY_i$. Hence its fill ratio is
\[
\frac{\mathbb{E}\left[\sum_{t=1}^T C_{it}\right]}
{\mathbb{E}\left[\sum_{t=1}^T \sum_{j=1}^K S_{itj}j\right]}
=
\frac{d\,\mathbb{E}[Y_i]}{d}
=
\mathbb{E}[Y_i].
\]
Thus any all-or-nothing policy satisfies
\[
\mathrm{FE\text{-}FR}
\le
\min_{i\in[T]} \mathbb{E}[Y_i]
\le
\frac{1}{T},
\]
where the last inequality follows from averaging.

This upper bound is attainable: before the horizon starts, choose one group uniformly at random from $[T]$ and fully serve that group if and when its request arrives, rejecting all other requests. Then each group is served with probability $1/T$, so the achieved FE-FR is exactly $1/T$.

Now consider Algorithm~\ref{alg:non-iid} with partial allocation. The load parameter is
\[
R
=
\frac{1}{K}\sum_{t=1}^T d
=
\frac{T}{K}\left(\frac{K}{2}+1\right)
=
\frac{T}{2}+T\epsilon.
\]
By Theorem~\ref{thm:non-iid}, Algorithm~\ref{alg:non-iid} achieves FE-FR $\frac{1}{1+R}$.

If $\epsilon \le \frac{1}{6}$, then
\[
1+R = 1+\frac{T}{2}+T\epsilon \le 1+\frac{2T}{3}.
\]
Since $T>3$, we have $1+\frac{2T}{3} < T$, and therefore
\(
\frac{1}{T} < \frac{1}{1+R}.
\)
This proves the strict separation.
\end{proof}

\subsection{Impossibility in the Time-Invariant Case}

We now present a simpler construction demonstrating that the separation persists even in the time-invariant setting with a single group and only two time periods.

\begin{example}\label{ex:all-or-nothing-general}
Let $K = \frac{1}{\epsilon}$ be the total available resource units, where $\epsilon \le 0.25$ is chosen such that $K$ is an even integer (i.e., $K \ge 4$). Consider a scenario with $T=2$ time slots and a single group. In each time slot, a request arrives with probability $1$, demanding exactly $d = \frac{K}{2} + 1 = \frac{1}{2\epsilon} + 1$ units. Consequently, the total demand over both time slots is deterministic and equals $\frac{1}{\epsilon} + 2$.
\end{example}

The following theorem establishes that Algorithm~\ref{alg:rcb} with partial allocation strictly outperforms all-or-nothing policies for this instance.

\begin{theorem}\label{thm:all-or-nothing-limit-iid}
In the scenario described in Example \ref{ex:all-or-nothing-general}, the maximum achievable FE-FR using an all-or-nothing allocation online policy is exactly $\frac{1}{2}$. In contrast, Algorithm \ref{alg:rcb} (RCB) achieves a strictly higher FE-FR for any $\epsilon \le 0.25$. Specifically:
\[
\text{FE-FR}_{\text{all-or-nothing}} = 0.5 \quad < \quad \text{FE-FR}_{\text{RCB}}.
\]
\end{theorem}

\begin{proof}
First, consider the all-or-nothing allocation rule. Since each request demands $d = \frac{K}{2} + 1$, satisfying two requests would require $2d = K + 2$ units, which exceeds the capacity $K$. Thus, any all-or-nothing policy can satisfy at most one request. With a total demand of $2d$, the maximum achievable FE-FR is:
\[
\text{FE-FR}_{\text{all-or-nothing}} = \frac{d}{2d} = \frac{1}{2}.
\]

Now, consider the partial allocation approach using Algorithm \ref{alg:rcb}. For this instance, the normalized load per slot is
\[
W = \frac{d}{K} = \frac{1}{2} + \epsilon,
\]
so the load parameter is $R = TW = 2W = 1 + 2\epsilon$. By Theorem \ref{thm:rcb-fe-fr}, the FE-FR achieved by RCB is lower-bounded by
\[
\text{FE-FR}_{\text{RCB}} \;\ge\; \frac{1 - e^{-R}}{R}.
\]
It therefore suffices to show that $\frac{1 - e^{-R}}{R} > \frac{1}{2}$ for $R = 1 + 2\epsilon$ with $\epsilon \le 0.25$.

Let $g(R) := \frac{1 - e^{-R}}{R}$. For $R>0$,
\[
g'(R) = \frac{(R+1)e^{-R} - 1}{R^2} < 0,
\]
where the inequality follows from $e^R > 1+R$ for all $R>0$. Hence $g$ is strictly decreasing on $(0,\infty)$.
Since $\epsilon \le 0.25$ implies $R = 1+2\epsilon \le \frac{3}{2}$, we have
\[
\frac{1 - e^{-R}}{R} \ge \frac{1 - e^{-3/2}}{3/2}>\frac{1}{2}.
\]
This proves the desired strict inequality.
\end{proof}

\section{Conclusion}

In this paper, we formalized the problem of Fair Online Resource Allocation with Indivisible Units (FORA-IU) and established optimal algorithmic frameworks for serving prioritized groups under uncertainty. By adopting the FE-FR-$\beta$ metric, we derived tight instance-dependent guarantees, achieving a fairness ratio of $\frac{1}{1+R_\beta}$ for general time-varying arrivals and a strictly stronger bound of $\frac{1 - (1 - R_\beta/T)^T}{R_\beta}$ (which is at least $\frac{1 - e^{-R_\beta}}{R_\beta}$) for stationary distributions. Moreover, our analysis revealed a theoretical separation between allocation protocols, demonstrating that policies restricted to all-or-nothing decisions are structurally incapable of matching the optimal fairness guarantees attainable through partial allocation.

Several meaningful directions remain for future investigation. First, relaxing the independence assumption to accommodate correlated or Markovian arrival processes would better capture the dynamics of real-world surges in demand. Second, while our current approach assumes known arrival probabilities, extending these results to the bandit setting, where distributions must be learned online, would significantly enhance practical applicability. Finally, investigating incentive-compatible mechanisms that prevent strategic over-reporting of demands while maintaining these fairness bounds presents an important challenge at the intersection of game theory and online optimization.

\begin{appendices}
\renewcommand{\theHequation}{A.\arabic{equation}}
\renewcommand{\theHfigure}{A.\arabic{figure}}
\renewcommand{\theHtable}{A.\arabic{table}}

\section{Supplemental material}

\subsection{Dynamic-programming computation of \texorpdfstring{$\gamma_{tj}$}{gamma_tj}}\label{sec:beta-calc}

\begin{algorithm}[H]
\caption{$\gamma$ Calculation}\label{alg:dynamic-beta}
\begin{algorithmic}[1]
\Require Parameters $K$, $N$, $T$ and $p_{itj}$ for $i \in [N]$, $t \in [T]$ and $j \in [K]$.
\Ensure $\gamma_{tj}$ for all $t \in [T]$ and $j \in [K]$.

\State Initialize $P(B_1=K)=1$ and $P(B_1=k)=0$ for all $k\in\{0,1,\dots,K-1\}$.
\State Initialize $\gamma_{1j}=1$ for all $j\in[K]$.

\For{$t = 2$ \textbf{to} $T$}
    \For{$k' = 1$ \textbf{to} $K$}
        \State Compute
        \begin{align}
           P(B_t = k')
           ={}& \sum_{k=k'+1}^{K}
           P(B_{t-1} = k)\cdot
           \frac{\sum_{i=1}^N p_{i(t-1)(k-k')}}{(1+R)\gamma_{(t-1)(k-k')}} \nonumber\\
           &+ P(B_{t-1} = k') \cdot \sum_{j = 1}^K
           \left(1-\frac{1}{(1+R)\gamma_{(t-1)j}} \right)
           \left(\sum_{i\in [N]}p_{i(t-1)j}\right)\nonumber\\
           &+ P(B_{t-1} = k') \cdot
           \left( 1 - \sum_{i=1}^N \sum_{j=1}^K p_{i(t-1)j} \right).
           \label{eq:dynamic-prob-calc}
        \end{align}
    \EndFor

    \State Compute $P(B_t = 0) = 1 - \sum_{k=1}^K P(B_t = k)$.

    \For{$j = 1$ \textbf{to} $K$}
        \State Compute
        \begin{align}\label{eq:dynamic-exp-calc}
            \gamma_{tj}
            =
            \mathbb{E}\left[\min(B_t / j, 1)\right]
            =
            \sum_{k=j}^{K} P(B_t = k)
            + \sum_{k=0}^{j-1} \frac{k}{j} P(B_t = k).
        \end{align}
    \EndFor
\EndFor
\end{algorithmic}
\end{algorithm}

\begin{prop}\label{prop:beta-correctness}
Algorithm~\ref{alg:dynamic-beta} correctly computes $\gamma_{tj} = \mathbb{E}[\min(B_t / j, 1)]$ for all $t \in [T]$, $j \in [K]$, and runs in $O(TN K^2)$ time.
\end{prop}

\begin{proof}
We prove correctness by induction on $t$.

For the base case $t=1$, the initial budget satisfies $B_1=K$ almost surely, so
\[
\gamma_{1j}=\mathbb{E}\left[\min(K/j,1)\right]=1
\qquad \text{for all } j\in[K].
\]

For the inductive step, assume the distribution of $B_{t-1}$ has already been computed correctly. To obtain $P(B_t=k')$, condition on the previous budget level and on the event in slot $t-1$.

If $k>k'$, then the budget can decrease from $k$ to $k'$ only if a request of demand $k-k'$ arrives and the policy accepts it. Under Algorithm~\ref{alg:non-iid}, that acceptance probability is
\[
\frac{1}{(1+R)\gamma_{(t-1)(k-k')}}.
\]
This gives the first term of \eqref{eq:dynamic-prob-calc}.

If the previous budget is already $k'$, then it remains $k'$ in either of two ways: a request arrives and is rejected, yielding the second term of \eqref{eq:dynamic-prob-calc}, or no request arrives, yielding the third term. These events are mutually exclusive and exhaustive, so \eqref{eq:dynamic-prob-calc} correctly computes the distribution of $B_t$.

Once $P(B_t=k)$ is known, \eqref{eq:dynamic-exp-calc} is exactly the definition
\[
\gamma_{tj}
=
\mathbb{E}\left[\min\left(\frac{B_t}{j},1\right)\right]
=
\sum_{k=0}^{K} P(B_t=k)\min\left(\frac{k}{j},1\right),
\]
split into the cases $k\ge j$ and $k<j$. This proves correctness.

For the running time, for each fixed $t$, computing each value $P(B_t=k')$ requires summing over at most $K$ previous budget states and, inside each term, aggregating over $N$ groups. Hence computing the full distribution of $B_t$ costs $O(NK^2)$. The subsequent computation of all $\gamma_{tj}$ values costs $O(K^2)$, which is dominated by $O(NK^2)$. Over all $T$ time steps, the total running time is therefore $O(TNK^2)$.
\end{proof}

\end{appendices}

\backmatter

\section*{Statements and Declarations}

\textbf{Competing interests.}
The authors declare that they have no competing interests.

\noindent \textbf{Data availability.}
No datasets were generated or analyzed during the current study.

\noindent \textbf{Author contributions.}
All authors contributed to the study conception, analysis, and manuscript preparation. All authors read and approved the final manuscript.

\noindent \textbf{Ethical approval.}
Not applicable. This article does not contain studies with human participants or animals performed by any of the authors.

\bibliography{jorsc_references}

\end{document}